\documentclass[11pt,reqno]{amsart}

\usepackage{hyperref,comment}
\usepackage{latexsym,amsmath, bm}
\usepackage{enumerate}
\usepackage{amsfonts}
\usepackage{amssymb}
\usepackage[margin=2.3cm]{geometry}
\usepackage{latexsym}
\usepackage{stmaryrd}
\usepackage{caption}
\usepackage[style=math,dashed=false,giveninits=true,maxbibnames=99]{biblatex}
\usepackage{fourier}
\usepackage{bbm}
\usepackage{color}
\usepackage{tikz}
\usetikzlibrary{decorations.pathreplacing, matrix}
\usepackage[ruled]{algorithm2e}
\usepackage{xparse}

\renewcommand{\epsilon}{\eps}

\newcommand{\J}{\vec J}

\newcommand\G{\vec G}
\newcommand\T{\vec T}

\numberwithin{equation}{section}

\renewcommand{\vec}[1]{\boldsymbol{#1}}

\newcommand\sa[1]{\textcolor{green!60!red}{#1}}

\newtheorem{definition}{Definition}[section]
\newtheorem{claim}[definition]{Claim}

\newtheorem{theorem}[definition]{Theorem}
\newtheorem{lemma}[definition]{Lemma}
\newtheorem{proposition}[definition]{Proposition}
\newtheorem{corollary}[definition]{Corollary}
\newtheorem{fact}[definition]{Fact}

\newtheorem{question}[definition]{Question}

\newcommand\cG{\mathcal{G}}

\newcommand\cT{\mathcal{T}}

\newcommand\cM{\mathcal{M}}

\newcommand\vB{\vec B}

\newcommand\w{{\omega}}
\newcommand{\beq}{\begin{equation}} \newcommand{\eeq}{\end{equation}}

\newcommand\eps{\varepsilon}

\newcommand\Erw{\mathbb{E}}

\newcommand{\Bin}{{\rm Bin}}

\newcommand\bc[1]{\left({#1}\right)}
\newcommand\cbc[1]{\left\{{#1}\right\}}

\newcommand\brk[1]{\left\lbrack{#1}\right\rbrack}

\newcommand\abs[1]{\left|{#1}\right|}

\newcommand{\whp}{a.a.s.}

\newcommand\pr{\mathbb{P}} 
\renewcommand\Pr{\pr} 

\newcommand\Lem{Lemma}
\newcommand\Prop{Proposition}
\newcommand\Thm{Theorem}

\newcommand{\galpha}{\cG_\alpha}
	\newcommand{\Buv}[2]{\vec{B}(#1, #2)}

	\NewDocumentCommand\gnp{gg}{%
		\ensuremath{\mathbb{G} (\IfNoValueTF{#1}{n}{#1}, \IfNoValueTF{#2}{p}{#2})}%
	}

\begin{document}

	\title{Random perturbation of sparse graphs}
	
	\author{Max Hahn-Klimroth, Giulia S. Maesaka, Yannick Mogge, Samuel Mohr, and Olaf Parczyk}
	\thanks{The research on this project was initiated during a workshop in Cuxhaven.
	We would like to thank the Hamburg University of Technology for their support.
	OP was supported by Technische Universit\"at Ilmenau, the Carls Zeiss Foundation, and DFG Grant PA 3513/1-1.
	MHK was supported by Stiftung Polytechnische Gesellschaft.
	SM was supported by DFG Grant 327533333.
	GSM is supported by the European Research Council (Consolidator Grant PEPCo 724903).}
	
	\address{Max Hahn-Klimroth, {\tt hahnklim@math.uni-frankfurt.de}, Goethe University, Mathematics Institute, 10 Robert Mayer St, Frankfurt 60325, Germany.}
	
	\address{Giulia Satiko Maesaka, {\tt giulia.maesaka@uni-hamburg.de}, Universit\"at Hamburg, Fachbereich Mathematik, 55 Bundesstr., Hamburg 20146, Germany.}
	
	\address{Yannick Mogge, {\tt yannick.mogge@tuhh.de}, Hamburg University of Technology, Mathematics Institute, 3 Am	Schwarzenberg-Campus, Hamburg 21073, Germany.}
	
	\address{Samuel Mohr, {\tt samuel.mohr@tu-ilmenau.de}, Ilmenau University of Technology, Mathematics Institute, 25 Weimarer St, Ilmenau 98693, Germany.}

	\address{Olaf Parczyk, {\tt o.parczyk@lse.ac.uk}, London School of Economics, Department of Mathematics, Houghton St, London, WC2A 2AE, UK.}
	
	\begin{abstract}
		In the model of randomly perturbed graphs we consider the union of a deterministic graph $\galpha$ with minimum degree $\alpha n$ and the binomial random graph $\gnp$.
		This model was introduced by Bohman, Frieze, and Martin and for Hamilton cycles their result bridges the gap between Dirac's theorem and the results by Pos\'{a} and Kor\v{s}unov on the threshold in $\gnp$.
		In this note we extend this result in $\galpha \cup \gnp$ to sparser graphs with $\alpha=o(1)$.
		More precisely, for any $\varepsilon>0$ and $\alpha \colon \mathbb{N} \mapsto (0,1)$ we show that a.a.s. $\galpha \cup \gnp{n}{\beta /n}$ is Hamiltonian, where $\beta = -(6 + \eps) \log(\alpha)$.
		If $\alpha>0$ is a fixed constant this gives the aforementioned result by Bohman, Frieze, and Martin and if $\alpha=O(1/n)$ the random part $\gnp$ is sufficient for a Hamilton cycle.
		We also discuss embeddings of bounded degree trees and other spanning structures in this model, which lead to interesting questions on almost spanning embeddings into $\gnp$.
	\end{abstract}
	
	\maketitle
	
	\section{Introduction and results}\label{Sec_intro}
	
	For $\alpha\in(0,1)$ we let $\galpha$ be an $n$-vertex graph with minimum degree $\delta(\galpha) \ge \alpha n$.
	A famous result by Dirac~\cite{Dirac} says that if $\alpha \ge 1/2$ and $n \ge 3$, then $\galpha$ contains a Hamilton cycle, i.e.~a spanning cycle through all vertices of $\galpha$.
	This motivated the more general questions of determining the smallest $\alpha$ such that $\galpha$ contains a given spanning structure.
	For example, there are results for trees~\cite{komlos2001spanning}, factors~\cite{hajnal1970proof}, powers of Hamilton cycles~\cite{komlos1998posa, komlos1998proof}, and general bounded degree graphs~\cite{bottcher2009proof}.
	This is a problem for deterministic graphs that belongs to the area of extremal graph theory.
	
	We can consider similar questions for random graphs, in particular, for the binomial random graph model $\gnp{n}{p}$, which is the probability space over $n$-vertex graphs with each edge being present with probability $p$ independent of all the others.
	Analogous to the smallest $\alpha$ we are looking for a function $\hat{p}=\hat{p}(n) \colon \mathbb{N} \mapsto (0,1)$ such that if $p=\omega(\hat{p})$ the probability that $\gnp{n}{p}$ contains some spanning subgraph tends to $1$ as $n$ tends to infinity and for $p=o(\hat{p})$ it tends to $0$.
	We call this $\hat{p}$ the threshold function  for the respective property (an easy sufficient criteria for its existence can be found in~\cite{bollobas1987threshold}) and if the first/second statement holds we say that $\gnp{n}{p}$ has/does not have this property asymptotically almost surely (a.a.s.). One often says that $\gnp{n}{p}$ undergoes a \textit{phase transition} at $\hat p$.
	For the Hamilton cycle problem Pos\'{a}~\cite{posa1976hamiltonian} and Kor\v{s}unov~\cite{korshunov1976solution} proved independently that $\hat{p}=\log n/n$ gives the threshold.
	Similar as above there was a tremendous amount of research on determining the thresholds for various spanning structures, e.g.~for matchings~\cite{erdHos1966existence}, trees~\cite{krivelevich2010embedding,montgomery2019spanning}, factors~\cite{johansson2008factors}, powers of Hamilton cycles~\cite{kuhn2012posa,nenadov2019powers}, and general bounded degree graphs~\cite{alon1992spanning,ferber2017embedding,ferber2018spanning,riordan2000spanning}.
	An extensive survey by B\"{o}ttcher can be found in~\cite{bottcher2017large}. 
	
	Motivated by the smoothed analysis of algorithms~\cite{spielman2004smoothed}, both these worlds were combined by Bohman, Frieze, and Martin~\cite{BFM03}.
	For any fixed $\alpha >0$, they defined the model of randomly perturbed graphs as the union $\galpha \cup \gnp{n}{p}$.
	They showed that $1/n$ is the threshold for a Hamilton cycle, meaning that there is a graph $\galpha$ such that with $p=o(1/n)$ there a.a.s.~is no Hamilton cycle in $\galpha \cup \gnp{n}{p}$ and for any $\galpha$ and $p=\omega(1/n)$ there a.a.s.~is a Hamilton cycle in $\galpha \cup \gnp{n}{p}$.
	It is important to note that in $\gnp$,  $p = 1/n$ is also the threshold for an almost spanning cycle, this is for any $\varepsilon>0$ a cycle on at least $(1-\varepsilon)n$ vertices.
	It should be further remarked that if $p=o(\log n/n)$ there are a.a.s.~isolated vertices in $\gnp{n}{p}$ and the purpose of $\galpha$ is to compensate for this and to help in turning the almost spanning cycle into a Hamilton cycle.
	
	This first result on randomly perturbed graphs~\cite{BFM03} sparked a lot of subsequent research on the thresholds of spanning structures in this randomly perturbed graphs model, e.g.~trees~\cite{bottcher2019universality,joos2018spanning,krivelevich2017bounded}, factors~\cite{balogh2019tilings}, powers of Hamilton cycles~\cite{bedenknecht2018powers,BMPPM18}, and general bounded degree graphs~\cite{BMPPM18}.
	As for a Hamilton cycle there is often a $\log$-factor difference to the thresholds in $\gnp{n}{p}$ alone, which is there for local reasons similar to isolated vertices.
	In most of these cases a $\galpha$, that is responsible for the lower bound, is the complete imbalanced bipartite graph $K_{\alpha n,(1-\alpha)n}$.
	In this model there are also results with lower bounds on $\alpha$~\cite{bennett2017adding,dudek2018powers,han2019tilings,nenadov2018sprinkling} and for Ramsey-type problems~\cite{das2019vertex,das2019ramsey}.

	\subsection{Hamiltonicity in randomly perturbed sparse graphs}\label{Sec_results}
	
	The aim of this note is to investigate a new direction.
	Instead of fixing an $\alpha \in (0,1)$ in advance we allow $\alpha$ to tend to zero with $n$.
	This extends the range of $\galpha$ to sparse graphs and we want to determine the threshold probability in $\galpha \cup \gnp$.
	For example, with $\alpha=1/\log n$ we have a sparse deterministic graph $\galpha$ with minimum degree $n/\log n$.
	Then $p=\omega(1/n)$ does not suffice in general, but it is sufficient to take $\galpha \cup \gnp{n}{\Theta(\log \log n)/n}$ to \whp~guarantee a Hamilton cycle.
	More generally, we can prove the following.
	
	\begin{theorem}\label{Thm_hamiltonicity}
		Let $\alpha = \alpha(n) : \mathbb{N} \mapsto (0,1)$ and $\beta = \beta(\alpha) = -(6+o(1)) \log(\alpha)$.
		Then a.a.s.~$\galpha \cup \gnp{n}{\beta /n}$ is Hamiltonian.
	\end{theorem}
	
	This extends the result of Bohman, Frieze, and Martin~\cite{BFM03} for constant $\alpha>0$.
	For even $n$ a direct consequence of this theorem is the existence of a perfect matching in the same graph.
	To prove Theorem~\ref{Thm_hamiltonicity} we use a result by Frieze~\cite{Frieze86} to find a very long path in $\gnp$ alone and then use the switching technique developed in~\cite{BMPPM18} to turn this into a Hamilton cycle.
	As it turns out, our method allows to prove the existence of a perfect matching with a slightly lower edge probability. 
	
	\begin{theorem}\label{Thm_PM}
		Let $\alpha = \alpha(n) : \mathbb{N} \mapsto (0,1)$ and $\beta = \beta(\alpha) = -(4+o(1)) \log(\alpha)$.
		Then a.a.s.~$\galpha \cup \gnp{n}{\beta /n}$ contains a perfect matching.
	\end{theorem}
	
	To see that in both theorems $\beta$ is optimal up to the constant factor, consider $\galpha=K_{\alpha n,(1-\alpha)n}$ and note that there cannot be a perfect matching, if we have more than $\alpha n$ isolated vertices on the $(1-\alpha)n$ side.
	The number of isolated vertices in $\gnp{n}{\beta /n}$ roughly is $n (1-\beta /n)^{n-1} \cong n \exp(-\beta)$, which is larger than $\alpha n$ if $\beta = o(-\log(\alpha))$.
	
	For proving results in the randomly perturbed graphs model good almost spanning results are essential.
	Typically, by almost spanning one means that for any $\eps>0$ we can embed the respective structure on at least $(1-\eps)n$ vertices.
	For paths and cycles in $\gnp{n}{C/n}$ this can, for example, be done using expansion properties and the DFS-algorithm~\cite{krivelevich2016long}.
	These almost spanning results are much easier than the spanning counterpart, because there is always a linear size set of available vertices.
	But for the proof of Theorem~\ref{Thm_hamiltonicity} this is not sufficient, because if $\alpha=o(1)$ we will not be able to take care of a linear sized leftover.
	Instead we exploit that we have $\gnp{n}{\beta/n}$ and use the following result showing that we can find a long cycle consisting of all but sublinearly many vertices.
	
	\begin{lemma}[Frieze~\cite{Frieze86}]
	\label{Lemma_almost_hamiltonian}
		Let $0 < \beta=\beta(n) \leq \log n$. Then $\gnp{n}{\beta/n}$ a.a.s.~contains a cycle of length at least $$\bc{1 - \bc{1-o(1)} \beta \exp\bc{-\beta}}n.$$
	\end{lemma}
	
	This is optimal, because this is asymptotically the size of the $2$-core (maximal subgraph with minimum degree $2$) of $\gnp$~\cite[Lemma~2.16]{frieze2016introduction}.
	A similar result holds for large matchings.
	
	\begin{lemma}[Frieze~\cite{Frieze86}]
	\label{Lemma_almost_pm}
		Let $0 < \beta=\beta(n) \leq \log n$. Then $\gnp{n}{\beta/n}$ a.a.s.~contains a matching consisting of at least $\bc{1 - \bc{1-o(1)} \exp\bc{-\beta}}n$ vertices.
	\end{lemma}
	
	Again this is optimal, because the number of isolated vertices is a.a.s.~$(1+o(1))e^{-\beta}n$~\cite[Theorem~3.1]{frieze2016introduction}.
	Observe, that also a bipartite variant of this lemma holds, which can be proved by removing small degree vertices and employing Halls theorem.
	
	\begin{lemma}
	    \label{Lemma_almost_pm_bip}
	    Let $0 < \beta=\beta(n) \leq \log n$. Then the bipartite binomial random graph $\gnp{n,n}{\beta/n}$ a.a.s.~contains a matching consisting of at least $\bc{1 - \bc{1-o(1)} \exp\bc{-\beta}}n$ edges.
	\end{lemma}
	
	\subsection{Bounded degree trees in randomly perturbed sparse graphs}
	
	After Hamilton cycles and perfect matchings, the next natural candidates are $n$-vertex trees with maximum degree bounded by a constant~$\Delta$.
	In $\gnp$ the threshold $\log n/n$ was determined in a breakthrough result by Montgomery~\cite{montgomery2019spanning}, in $\galpha$ it is enough to have a fixed $\alpha>1/2$~\cite{komlos1995proof}, and in $\galpha \cup \gnp$ with constant $\alpha>0$ the threshold is $1/n$~\cite{krivelevich2017bounded}.
	To obtain a result similar to Theorem~\ref{Thm_hamiltonicity} for bounded degree trees using our approach we need an almost spanning result similar to Lemma~\ref{Lemma_almost_hamiltonian}.
	With a similar approach as for Theorem~\ref{Thm_hamiltonicity} and~\ref{Thm_PM} we obtain the following modular statement.
	
	\begin{theorem}\label{Thm_trees}
    Let $\Delta \ge 2$ be an integer and suppose that $\alpha,\beta,\eps \colon \mathbb{N} \mapsto [0,1]$ are such that $4 (\Delta+1) \eps < \alpha^{\Delta + 1}$ and a.a.s.~$\gnp{n}{\beta/n}$ contains a given tree with maximum degree $\Delta$ on $(1 - \eps)n$ vertices.
    Then any tree with maximum degree $\Delta$ on $n$ vertices is a.a.s.~contained in the union $\galpha \cup \gnp{n}{\beta/n}$.
    \end{theorem}
	
	Next we discuss the almost spanning results that we can obtain in the relevant regime.
	Improving on a result of Alon, Krivelevich, and Sudakov~\cite{alon2007embedding}, Balogh, Csaba, Pei, and Samotij~\cite{balogh2010large} proved that for $\Delta \ge 2$ there exists a $C>0$  such that for $\eps>0$ a.a.s.~$\gnp{n}{\beta/n}$ contains any tree with maximum degree $\Delta$ on at most $(1-\eps)n$ vertices provided that $\beta \ge \tfrac{C}{\eps} \log \tfrac{1}{\eps} $.
	For the proof they only require that the graph satisfies certain expander properties.
	This can be extended to the range where $\eps \to 0$ and $\omega(1)=\beta \le \log n$ and following along the lines of their argument we get the following.
    
	\begin{lemma}
	    \label{lem:almost_tree}
	    For $\Delta \ge 2$ there exists a $C>0$ such that for any $0 < \beta=\beta(n) \leq \log n$ and $\eps=\eps(n)>0$ with $\beta \ge \tfrac{C}{\eps} \log \tfrac{1}{\eps}$ the following holds.
	    $\gnp{n}{\beta/n}$ a.a.s.~contains any bounded degree tree on at most $\bc{1 - \eps}n$ vertices.
	\end{lemma}
	
	Then together with Theorem~\ref{Thm_trees} we obtain the following.
	
	\begin{corollary}
	\label{cor:trees}
	For $\Delta \ge 2$ there exists a $C>0$ such that for $\alpha = \alpha(n) : \mathbb{N} \mapsto (0,1)$ and $\beta = \beta(\alpha) = C \alpha^{-(\Delta+1)} \log \tfrac{1}{\alpha}$ the following holds.
	Any $n$-vertex tree $T$ with maximum degree $\Delta$ is a.a.s.~contained in $\galpha \cup \gnp{n}{\beta /n}$.
	\end{corollary}
	
	The proof for the dense case in~\cite{krivelevich2017bounded} uses regularity and it is unlikely to give anything better in the sparse regime.
	As remarked in~\cite{alon2007embedding} the condition on the almost spanning embedding in $\gnp{n}{\beta/n}$ could possibly be improved to $\beta > \log \tfrac{C}{\eps}$, then covering almost all non-isolated vertices.
	More precisely this asks for the following.
	
	\begin{question}
	    \label{que:tree}
		For every integer $\Delta$ there exists $C>0$ such that with $0<\beta=\beta(n)\le \log n$ the following holds.
		Is any given tree with maximum degree $\Delta$ on $$(1-C \exp(-\beta))n$$ vertices a.a.s.~contained in $\gnp{n}{\beta/n}$?
	\end{question}
	
	With Theorem~\ref{Thm_trees} this would then give that already $\beta=-(\Delta+1)\log (C \alpha)$ suffices, which would be optimal up to the constant factors.
	We want to briefly argue why it is possible to answer this question for large families of trees and what the difficulties are.
	For simplicity we only discuss the case $\beta=\log \log n$ and note that by Lemma~\ref{lem:almost_tree} above we can embed trees on roughly $(1-1/\log \log n)n$ vertices.
	A very helpful result for handling trees by Krivelevich~\cite{krivelevich2010embedding} states that for any integer $n,k>2$, a tree on $n$ vertices either has at least $n/4k$ leaves or a collection of at least $n/4k$ bare paths (internal vertices of the path have degree $2$ in the tree) of length $k$.
	If there are at least $n /(4\log \log n)$ leaves, we can embed the tree obtained after removing the leaves.
	Then we can use a fresh random graph and Lemma~\ref{Lemma_almost_pm_bip} to find a matching for all the leaves, completing the embedding of the tree.
	
	On the other hand, if there are at least $n \log \log n/(4 \log n)$ bare paths of length $\log n/\log \log n$, it is possible to embed all but $n /\log n$ of these paths, which are all but $n /\log \log n$ vertices.
	Then one has to connect the remaining paths, again using ideas from~\cite{montgomery2019spanning}.
	In between both cases it is not clear what should be done, because we might have $n / \log n$ leaves and $n/(4\log \log n)$ bare paths of length $\log \log n$.
	The length of the paths are too short to connect them and the leaves are too few for the above argument.
    Answering this questions and thereby improving the result of Alon, Krivelevich, and Sudakov~\cite{alon2007embedding} is a challenging open problem.
	
	\subsection{Other spanning structures}
	
	As mentioned above, embeddings of spanning structures in $\galpha$, $\gnp{n}{p}$, and $\galpha \cup \gnp{n}{p}$ for fixed $\alpha>0$ have also been studied for other graphs such as powers of Hamilton cycles, factors, and general bounded degree graphs.
	In most of these cases almost spanning embeddings (e.g.~Ferber, Luh, and Nguyen~\cite{ferber2017embedding}) can be generalised such that previous proofs can be extended to the regime $\alpha=o(1)$ with $\beta=\alpha^{-1/C}$, similar to what we do in Corollary~\ref{cor:trees}.
	Further improvements seem to be hard, because better almost spanning results are similar in difficulty to spanning results in $\gnp{n}{p}$ alone.
	We want to discuss this on one basic example, the triangle factor, which is the disjoint union of $n/3$ triangles.
	
	In $\galpha$ we need $\alpha \ge 2/3$, in $\gnp$ the threshold is $n^{-2/3}\log^{1/3}n $, and in $\galpha \cup \gnp$ with a fixed $\alpha>0$ it is $n^{-2/3}$.
	Note that the $\log$-term in $\gnp$ is needed to ensure that every vertex is contained in a triangle, which is essential for a triangle factor.
	Using Janson's inequality~\cite[Theorem~21.12]{frieze2016introduction} it is not hard to prove the almost spanning result for a triangle factor on at least $(1-\varepsilon)n$ vertices with $p=\omega(n^{-2/3})$. 
	This can be generalised to $\gnp{n}{\beta n^{-2/3}}$ giving a.a.s.~a triangle factor on at least $(1-C/\beta)n$ vertices.
	Again, this can only give something with $\beta=\alpha^{-1/C}$ in $\galpha \cup \gnp{n}{\beta n^{-2/3}}$ and to improve this we ask the following.
	\begin{question}
		Let $0<\beta=\beta(n)\le \log^{1/3} n$. Does $\gnp{n}{\beta n^{-2/3}}$ a.a.s.~contain a triangle factor on at least $$\left(1- (1-o(1))\exp(-\beta^3)\right)n$$ vertices?
	\end{question}
	Observe, that this is a.a.s.~the number of vertices of $\gnp{n}{\beta n^{-2/3}}$ that are not contained in a triangle.
	Similar questions for other factors or more general structures would be of interest.
	It took a long time until Johannson, Kahn, and Vu~\cite{johansson2008factors} determined the threshold for the triangle factor.
	This conjecture seems to be of similar difficulty, whereas for our purposes it would already be great to obtain a triangle factor on at least $(1-C \exp(-\beta^3))n$ vertices for some $C>1$.
	
	For the remainder of this note we prove Theorem~\ref{Thm_hamiltonicity} and~\ref{Thm_trees} in Section~\ref{Sec_hamil_proof} and~\ref{Sec_trees_proof} respectively.

	\section{Hamiltonicity}\label{Sec_hamil_proof}
	
	We will prove the following proposition that will be sufficient to prove the theorem together with known results on Hamilton cycles in $\gnp$.
	
	\begin{proposition}
	\label{Prop_hamiltonicity_large}
	Let $\alpha = \alpha(n) : \mathbb{N} \mapsto (0,1)$ such that $\alpha = \omega(n^{-1/6})$, and let $\beta = \beta(\alpha) = -(6+o(1)) \log(\alpha)$. 
	Then a.a.s. $\galpha \cup \gnp{n}{\beta /n}$ is Hamiltonian.
	\end{proposition}
	
	\subsection*{Proof of Theorem~\ref{Thm_hamiltonicity}}
	Let $\alpha, \beta>0$ such that $\beta = -(6+o(1)) \log(\alpha)$.
	If $\alpha = O(n^{-1/6})$, we have $\beta \ge (1+o(1)) \log n$ and we can infer that a.a.s.~there is a Hamilton cycle in $G(n,\beta/n)$ (this follows from an improvement on the result concerning the threshold for Hamiltonicity~\cite{komlos1983limit}).
	On the other hand, if $\alpha = \omega(n^{-1/6})$, then we apply Proposition~\ref{Prop_hamiltonicity_large} to a.a.s.~get the Hamilton cycle.
    \qed

\subsection*{Proof of \Prop~\ref{Prop_hamiltonicity_large}}

    To prove the proposition we apply the following strategy.
	We first find a long path in $\gnp$ alone.
	Then, by considering the union with $\galpha$, we obtain a reservoir structure for each vertex that allows us to extend the length of the path iteratively.
	Finally, we will also be able to close this path to a cycle on all vertices.
	W.l.o.g.~we can assume that $\alpha<1/10$.

	\subsection*{Finding a long path} 
	
	Let $P=p_1, \dots, p_\ell$ be the longest path that we  can find in $\cG_1=\gnp{n}{(\beta-1)/n}$ and let $V' = \cbc{ v_1,\dots,v_k } =  V(\cG_1) \setminus \cbc{ p_1, ..., p_\ell }$ be the left-over.
	Then, by Lemma~\ref{Lemma_almost_hamiltonian}, we get a.a.s.~that
	\begin{align}
	k=\abs{V'}=n-\ell \le \bc{1-o(1)} \beta \exp\bc{1-\beta}n. \label{Eq_SizeLeftover}
	\end{align}
	Next, let $P'$ be a collection of vertices of $P$, where we take every other vertex, excluding the last, that is 
	\begin{align}
	\label{Def_Pprime} P' = \cbc{ p_i : i \equiv 0 \pmod 2 } \setminus \cbc{ p_\ell}
	\end{align}
	In the following, we will work on $P'$ instead of all of $P$, ensuring that certain absorbing structures do not overlap.
	
	\subsection*{Absorbing the left-over} 
	We now consider the union $\galpha \cup \cG_1$.
	The following absorbing structure is the key to the argument.
	\begin{definition}
	    \label{Def_absorber}
		For any vertices $u, v \in V(\galpha \cup \cG_1)$ let 
		\begin{align}
		\label{Def_Buv} \Buv{u}{v} = \cbc{ x \in N_{\galpha}(u) \cap P' \mid N_P(x) \subseteq N_{\galpha}(v)  }.
		\end{align}
	\end{definition}
	
	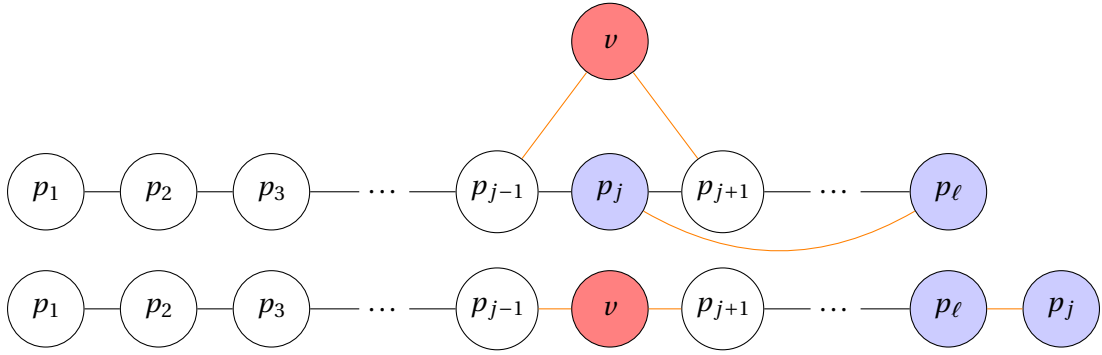
\begin{figure}
	\captionsetup{margin=0.5cm}
		\centering
		\begin{tikzpicture}[scale=1]
		\node[circle, fill=none, draw, minimum size=3em] (1) at (0,0) {$p_1$};
		\node[circle, fill=none, draw, minimum size=3em] (2) at (1.5, 0)  {$p_2$};
		\node[circle, fill=none, draw, minimum size=3em] (3) at (3, 0) {$p_3$};
		\node (4) at (4.5, 0) {\dots};
		\node[circle, fill=none, draw, minimum size=3em] (5) at (6,0) {$p_{j-1}$};
		\node[circle, fill=blue!20, draw, minimum size=3em] (6) at (7.5, 0)  {$p_j$};
		\node[circle, fill=none, draw, minimum size=3em] (7) at (9, 0) {$p_{j+1}$};
		\node (8) at (10.5, 0) {\dots};
		\node[circle, fill=blue!20, minimum size=3em, draw] (9) at (12, 0)  {$p_\ell$};
		\node[circle, fill=red!50, draw, minimum size=3em, draw] (10) at (7.5,2) {$v$};
		
		\draw (1) -- (2) -- (3) -- (4) -- (5) -- (6) -- (7) -- (8) -- (9);
		\path[-] (10)  edge   [orange]  (5);
		\path[-] (10)  edge   [orange]  (7);
		\path[-] (9)  edge   [orange, bend left=30]  (6);

\node (anchor) at (14.1,0) {};
		\end{tikzpicture}
		
		\begin{tikzpicture}[scale=1]
		\node[circle, fill=none, draw, minimum size=3em] (1) at (0,0) {$p_1$};
		\node[circle, fill=none, draw, minimum size=3em] (2) at (1.5, 0)  {$p_2$};
		\node[circle, fill=none, draw, minimum size=3em] (3) at (3, 0) {$p_3$};
		\node (4) at (4.5, 0) {\dots};
		\node[circle, fill=none, draw, minimum size=3em] (5) at (6,0) {$p_{j-1}$};
		\node[circle, fill=blue!20, draw, minimum size=3em] (6) at (13.5, 0)  {$p_j$};
		\node[circle, fill=none, draw, minimum size=3em] (7) at (9, 0) {$p_{j+1}$};
		\node (8) at (10.5, 0) {\dots};
		\node[circle, fill=blue!20, minimum size=3em, draw] (9) at (12, 0)  {$p_\ell$};
		\node[circle, fill=red!50, draw, minimum size=3em, draw] (10) at (7.5,0) {$v$};
		
		\draw (1) -- (2) -- (3) -- (4) -- (5);
		\draw (7) -- (8) -- (9);
		\path[-] (10)  edge   [orange]  (5);
		\path[-] (10)  edge   [orange]  (7);
		\path[-] (9)  edge   [orange]  (6);

        \node (anchor) at (14.1,0) {};
		\end{tikzpicture}
		
		\caption{The top shows a path $P = p_1, \dots, p_\ell$ and the left-over vertex $v$. Black edges belong to the random graph, orange edges can be found in $\galpha$. The bottom shows the situation after absorbing $v$ using that $p_j \in \Buv{p_\ell}{v}$. }
		\label{Fig_AbsorberBuv}
	\end{figure}
	
	If for some $v \in V'$ there is an $p_j \in \Buv{p_\ell}{v}$ we can proceed as follows (see Figure \ref{Fig_AbsorberBuv}).
	By definition we have $p_{j-1}, p_{j+1} \in N_{\galpha(v)}$ and $p_{j} \in N_{\galpha}(p_\ell) \cap P$.
	Then $p_j$ can be replaced by $v$ in the path $P$ and can now be appended to the path $P$ at $p_\ell$.
	So we get the path $\tilde{P} = p_1, \dots, p_{j-1},v, p_{j+1}, \dots, p_\ell, p_j$, where $\tilde{P} \subset P \cup \galpha$. 

    To iterate this argument we show that a.a.s.~for any pair of vertices $u$ and $v$, the set $\Buv{u}{v}$ is large enough.
	
	\begin{claim}\label{Lemma_SizeOfBuv}
		We have a.a.s.~$\abs{\Buv{u}{v}} \geq \alpha^3 n / 4$ for any $u,v \in V(\galpha \cup \cG_1)$.
	\end{claim}
	
	\begin{proof}
		Let $u, v$ be arbitrary vertices in $V=V(\galpha \cup \cG_1)$. The set $\Buv{u}{v}$ is uniformly distributed over $P'$, because $\gnp{n}{(\beta-1)/n}$ is sampled independently of the deterministic graph $\galpha$. Then by definition 
		\begin{align}
		\Erw \brk{\abs{\Buv{u}{v}}} \ge \frac{9}{10} \alpha^3 \abs{P'} \geq \frac{2}{5} \alpha^3 \bc{1 - \bc{1-o(1)} \beta \exp\bc{1-\beta}} n \geq \alpha^3 n / 3. \label{Eq_ErwBuv}
		\end{align}
		An immediate consequence of $\Buv{u}{v}$ being uniformly settled over $\gnp{n}{(\beta-1)/n}$ is that $\abs{\Buv{u}{v}} \sim \Bin(\abs{P'}, \alpha^3)$. It follows from \eqref{Eq_ErwBuv} and the Chernoff bound that there is a sufficiently small, but constant, $\delta > 0$ s.t.
		\begin{align}
		\Pr \bc{\abs{\Buv{u}{v}} < \alpha^3 n/4 } \leq \Pr \bc{\abs{\Buv{u}{v}} < (1- \delta)\Erw \brk{\abs{\Buv{u}{v}}} } \leq \exp \bc{ - \delta^2 / 8 \alpha^3 n } < \exp \bc{- \sqrt{n} }. \label{Eq_ConcBuv}
		\end{align}
	 The lemma follows from a union bound over all $\binom{n}{2}$ choices for $u,v$ and \eqref{Eq_ConcBuv}.
	\end{proof}
	
	We now have everything at hand to absorb all but two of the left-over vertices $v \in V'$ onto a path of length $n-2$. We do this inductively using Algorithm \ref{Algo_IncreaseP}.
	
	\IncMargin{1em}
	\begin{algorithm}
		\SetAlgoLined
		\SetKwData{Left}{left}\SetKwData{This}{this}\SetKwData{Up}{up}
		\SetKwFunction{Union}{Union}\SetKwFunction{FindCompress}{FindCompress}
		\SetKwInOut{Input}{Input}\SetKwInOut{Output}{Output}
		\Input{Path $P = p_1 \dots p_\ell$, set of left-over vertices $V' = \cbc{v_1, \dots, v_k}$.}
		\Output{Path $\tilde{P}$ in $P \cup \galpha$ on $n-2$ vertices.}
		\BlankLine
		Define $\ell_1 = \ell$, $P_1 = P$ with $P_1 = u^1_1 \dots  u^1_{\ell_1}$\; 
		Define for any $u, v$ the set $B_1(u,v)=\Buv{u}{v}$\;
		Define $V_1' = V'$\;
		\For{$i = 1$ \KwTo $k-2$}{
			Choose $u^i_j \in B_i(u_{\ell_i}^i, v_i)$ and absorb $v_i$ onto $P_i$\;
			Denote by $P_{i+1} = u^i_1 \dots u^i_{j-1} v_i u^i_{j+1} \cdots u^i_{\ell_i} u^i_j = u^{i+1}_1 \dots u^{i+1}_{\ell_{i}+1}$ the resulting path\;
			Update $\ell_{i+1} = \ell_i + 1, V_{i+1}' = V_i' \setminus \cbc{v_i}$\;
			Set $B_{i+1}(u, v) = B_{i}(u, v) \setminus \cbc{ u^i_{j}}$ for any $u, v$\;
		}
		$\tilde{P} = P_{k}$\;
		\caption{Absorbs all but two vertices of the left-over set $V'$ onto a path.} 
		\label{Algo_IncreaseP}
	\end{algorithm}\DecMargin{1em}
	
Let $\tilde{P}, B_i(\cdot, \cdot)$ be defined as in Algorithm \ref{Algo_IncreaseP}. In order to see that the algorithm terminates with $\tilde{P} = P_{k}$ it suffices to prove, that $B_i(u,v)$ is not empty for any $u,v \in V$ and $i=1\dots k$.
    By definition of $P'$ in~\eqref{Def_Pprime} we have $|\Buv{u}{v} \setminus B_i(u,v)| \le i$ and using
	Claim~\ref{Lemma_SizeOfBuv} and \eqref{Eq_SizeLeftover} we get
	\begin{align}
	   \label{Eq_LeftOverVertices}
	    |B_i(u,v)| \ge \alpha^3 n/8,
	\end{align}
	whenever $\beta \exp \bc{1- \beta} < \alpha^3 / 8$.
	As this holds by definition of $\beta = -(6+o(1)) \log(\alpha)$ and with $\alpha < 1/10$, we get that~\eqref{Eq_LeftOverVertices} holds for all $u,v \in V$ and any $i=1,\dots,k$.
	
\subsection*{Closing the cycle}

	We have found a path $\tilde{P} = p_1, \dots, p_{n-2}$ and we are left with two vertices $v_{k-1}, v_{k}$ that are not on the path.
	It is possible to close the Hamilton cycle by absorbing $v_{k-1}$ and $v_{k}$ if there is an edge between $A := B_k(p_1,v_{k-1})$ and $B := B_k(p_{n-2},v_k)$.
	Indeed, we then have w.l.o.g.~$i<j$ such that $p_i \in A$, $p_j \in B$, and there is an edge $p_ip_j$.
	By definition of $A$ and $B$ we can then obtain the Hamilton cycle
	$$p_i,p_1,\dots,p_{i-1},v_{k-1},p_{i+1},\dots,p_{j-1},v_k,p_{j+1},\dots,p_{n-2},p_j.$$
	
	It remains to prove that we have an edge between $A$ and $B$.
	For this we reveal $\cG_2 = \gnp{n}{1/n}$.
	As $|A|,|B| \ge \alpha^3 n/8$ by~\eqref{Eq_LeftOverVertices} we get
	\begin{align}
		\Erw \brk{ e_{\cG_2} \bc{A, B} } \geq \frac{1}{n} \cdot \bc{\frac{\alpha^3 n}{16}}^2 = \omega(1),
	\end{align}
	as $\alpha = \omega (n^{-1/6})$.
	Together with Chernoff's inequality this implies that a.a.s~$e_{\cG_2} \bc{A, B}>0$.
	As the union of $\cG_1$ and $\cG_2$ can be coupled as a subgraph of $\gnp{n}{\beta/n}$ this implies that a.a.s.~there is a Hamilton cycle in $\galpha \cup \gnp$ and finishes the proof of \Prop~\ref{Prop_hamiltonicity_large}. \qed

Observe, that when running the same proof for Theorem~\ref{Thm_PM} we can obtain the better constant by adapting the definition of the $\Buv{u}{v}$ to the setup of perfect matchings and then proving that a.a.s.~$|\Buv{u}{v}| \ge \alpha^2n/4$.
We spare the details here.

\section{Bounded degree trees}
\label{Sec_trees_proof}

\Thm~\ref{Thm_trees} is modular, which turns almost spanning embeddings in the random graph into spanning embeddings in the union $\galpha \cup \gnp{n}{\beta/n}$.
The proof is very similar to the proof for Hamilton cycles and we will spare some details.

\subsection*{Proof of \Thm~\ref{Thm_trees}}
Let $\galpha$ be given and $\cG = \gnp{n}{\beta/n}$.
Let $\cT$ be an arbitrary tree on $n$ vertices with maximum degree $\Delta$.
Denote by $\cT_{\eps}$ the tree obtained from $\cT$ by the following construction.
\begin{enumerate}
    \item Set $\cT_0 = \cT$.
    \item In every step $i$, check whether $\cT_i$ has at most $(1 - \eps)n$ vertices. 
    \begin{itemize}
        \item If this is the case, set $\cT_\eps = \cT_i$ and finish the process.
        \item Otherwise, create $\cT_{i+1}$ by deleting one leaf of $\cT_i$. 
    \end{itemize}
\end{enumerate}
We denote by $L$ the left-over, that are the vertices removed during construction of $\cT_{\eps}$.
Then
\begin{align*}
\abs{V(\cT_{\eps})} \leq (1 - \eps)n, \qquad \abs{L} \leq \eps n+1, \qquad \text{and} \qquad V(\cT) = V(\cT_{\eps}) \cup L.
\end{align*}
Next we let $T$ be an independent subset of the vertices of $\cT_\eps$ such that the vertices in $T$ do not have neighbours outside of $\cT_\eps$ with respect to $\cT$.
Observe, that there exists such a $T$ such that $|T| \ge \frac{(1-\Delta \eps)n}{\Delta+1}$.

By assumption we a.a.s.~have an embedding $\cT_\eps'$ of $\cT_\eps$ into $\cG$ and we denote by $T'$ the image of $T$ under this embedding.
We adapt Definition~\ref{Def_absorber} and define for any two vertices $u,v$
$$\vB(u,v) = \cbc{ x \in N_{\galpha}(u) \cap T' \mid N_{\cT_\eps'}(x) \subset N_{\galpha (v)}}.$$
As before, if we want to embed a vertex $w$ that is a neighbour of an already embedded vertex $u$ in $\cT_\eps$ and $v$ is an available vertex we can do it if $\vB(u,v)$ is non-empty.
More precisely, with $x \in \vB(u,v)$, we can embed the vertex embedded onto $x$ to $v$, embed $w$ to $x$, and obtain a valid embedding of $\cT_\eps$ with an additional neighbour of $u$.
Analogous to Claim~\ref{Lemma_SizeOfBuv} we get the following.

\begin{claim}
\label{Claim_Buv}
We have a.a.s.~$\abs{\Buv{u}{v}} \geq \frac{\alpha^{\Delta+1} n}{4(\Delta+1)}$ for any $u,v \in V(\galpha \cup \cG)$.
\end{claim}

Therefore, similar to Algorithm~\ref{Algo_IncreaseP}, we can iteratively append leaves to $\cT_\eps$ to obtain an embedding of $\cT$ into $\galpha \cup \cG$.
As in every step we lose at most one vertex from each $\vB(u,v)$ this works as long as
$$\abs{L} \leq \eps n+1 < \abs{\vB(u,v)},$$
which holds by Claim~\ref{Claim_Buv} and the assumption on $\eps$ and $\alpha$.
\qed

\printbibliography
\end{document}